\documentclass[reqno,11pt]{amsart}

\usepackage{amssymb,amsthm}
\usepackage{amsmath}
\usepackage{amsfonts}
\usepackage{dsfont}

\usepackage[a4paper,  margin=2.6cm]{geometry}

\usepackage[active]{srcltx}
\makeatletter\@addtoreset{equation}{section}\makeatother

\newtheorem{theorem}{Theorem}[section]

\newtheorem{lemma}[theorem]{Lemma}

\newtheorem{proposition}[theorem]{Proposition}
\newtheorem{assumption}[theorem]{Assumption}
\newtheorem{definition}[theorem]{Definition}
\newtheorem{remark}[theorem]{Remark}
\numberwithin{equation}{section}

\title[Uniqueness of Markov random fields with higher-order dependencies]{Uniqueness of
Markov random fields with higher-order dependencies}

\author{Dorota K\c{e}pa-Maksymowicz}

\address{Instytut Matematyki, Uniwersytet Marii Curie-Sk{\l}odowskiej, 20-031 Lublin, Poland}

\email{dorota.kepa-maksymowicz@mail.umcs.pl}

\author{ Yuri  Kozitsky}

\address{Instytut Matematyki, Uniwersytet Marii Curie-Sk{\l}odowskiej, 20-031 Lublin, Poland}
\email{jurij.kozicki@mail.umcs.pl}



\begin{document}

\maketitle

\begin{abstract}
Markov random fields on a countable set $\sf V$ are studied. They
are canonically set by a specification $\gamma$, for which the
dependence structure is defined by  a pre-modification $(h_e)_{e\in
{\sf E}}$  -- a consistent family of functions $h_e : S^e\to
[0,+\infty)$, where $S$ is a standard Borel space and $\sf E$ is an
infinite collection of finite $e\subset {\sf V}$. Different $e$ may
contain distinct number of elements, which, in particular, means
that the dependence graph ${\sf H}=({\sf V}, {\sf E})$ is a
hypergraph. Given $e\in {\sf E}$, let $\delta (e)$ be the
logarithmic oscillation of $h_e$. The result of this work is the
assertion that the set of all fields $\mathcal{G}(\gamma)$ is a
singleton whenever $\delta(e)$ satisfies a condition, a particular
version of which can be $\delta(e) \leq \varkappa g(n_{\sf L}(e))$,
holding for all $e$ and some $\sf H$-specific $\varkappa\in (0,1)$.
Here $g$ is an increasing function, e.g., $g(n) = a+\log n$, and
$n_{\sf L}(e)$ is the degree of $e$ in the line-graph ${\sf L}({\sf
H})$, which may grow ad infinitum. This uniqueness condition is
essentially less restrictive than those based on the classical
Dobrushin uniqueness theorem, according to which either of $|e|$,
$n_{\sf L}(e)$ and $\delta(e)$ should be globally bounded. We also
prove that its fulfilment implies that the unique element of
$\mathcal{G}(\gamma)$ is globally Markov.

\end{abstract}

\section{Setup}

There exists a permanent interest to  random fields based on
discrete structures, which can be explained by their numerous
applications in mathematical physics, spatial statistics, image
analysis, and many other sciences, see, e.g.,
\cite{AHK,b2,DS,DS1,Dorlas,GeM,Li,Nor,Ny,Preston,Tj}. A systematic
presentation of the the theory of such fields may be found in
\cite{Ge,Ny,Preston}. In this work, we mostly use Georgii's
monograph \cite{Ge} as the source of notions, facts and notations on
this subject.

Let $\sf V$ be a countable set and $(S, \mathcal{S})$ a measurable
space. A random field on $\sf V$ is a collection of random variables
$(\sigma_x)_{x\in {\sf V}}$ (also called \emph{spins}) defined on
some probability space that take values in $S$ (\emph{single-spin
space}). In the canonical realization, a random field is a
probability measure on $(\Sigma, \mathcal{F})$, where $\Sigma =
S^{\sf V}$ and $\mathcal{F}=\mathcal{S}^{\sf V}$. Typically, the
dependence type of a random field is \emph{specified} by a family
$\gamma=(\gamma_\Lambda)_{\Lambda \in \mathcal{V}}$ of probability
kernels, where $\mathcal{V}$ is the collection of all nonempty
finite subsets $\Lambda \subset {\sf V}$.  For $\Delta \subset {\sf
V}$, let $\mathcal{F}_\Delta$ be the sub-$\sigma$-algebra of
$\mathcal{F}$ generated by the maps $\Sigma \mapsto \sigma_\Delta
=(\sigma_x)_{x\in \Delta}\in S^\Delta$. Then for $\Lambda \in
\mathcal{V}$, the external $\sigma$-algebra  of events outside
$\Lambda$ is $\mathcal{F}_{\Lambda^c}$, $\Lambda^c:={\sf V}\setminus
\Lambda$. A probability measure on $(\Sigma, \mathcal{F})$ is said
to be specified by $\gamma=(\gamma_\Lambda)_{\Lambda \in
\mathcal{V}}$ if it satisfies the condition
\begin{equation}
  \label{G1}
  \mu(A|\mathcal{F}_{\Lambda^c}) = \gamma_\Lambda (A|\cdot) \qquad
  \mu-{\rm almost} \ {\rm surely},
\end{equation}
holding for all $A\in \mathcal{F}$ and $\Lambda \in \mathcal{V}$.
The set of all $\mu$ satisfying (\ref{G1}) is denoted by
$\mathcal{G}(\gamma)$. It is known, see \cite[Proposition 1.24, page
17]{Ge}, that a given $\mu$ belongs to $\mathcal{G}(\gamma)$ if and
only if it solves
\begin{equation}
  \label{G1a}
\mu(A)=  (\mu \gamma_\Lambda)(A) := \int_{\Sigma} \gamma_\Lambda
(A|\sigma)\mu(d \sigma), \qquad A \in \mathcal{F}, \quad \Lambda \in
\mathcal{V},
\end{equation}
known as the Dobrushin-Lanford-Ruelle equation.

In most of the cases, single-spin space $S$ is finite, see e.g.,
\cite{BD,vdBM,CM,GeM,H}. In this work, $(S, \mathcal{S})$ is a
standard Borel space, which means that it is measurably isomorphic
to a complete and separable metric space. An extension of the theory
to nonseparable single-spin spaces can be found in \cite{Dorlas}.
For $(S, \mathcal{S})$ being a standard Borel space, the existence
of measures satisfying (\ref{G1}), (\ref{G1a}) follows by rather
standard arguments, whereas proving uniqueness is a much more
nontrivial accomplishment, even if $S$ is finite.

Let $\chi$ be a probability measure on $(S, \mathcal{S})$. For
$\Lambda \in \mathcal{V}$, by $\chi^\Lambda$ we denote the
corresponding product measure on $(S^\Lambda,
\mathcal{S}^{\Lambda})$. Then the {\it independent} specification
consists of the kernels
\begin{equation}
  \label{G2}
  \gamma^\chi_\Lambda (\cdot| \sigma) = \chi_\Lambda (\cdot|\sigma):= ( \chi^{\Lambda} \times
  \delta_{\sigma_{\Lambda^c}} )(\cdot),
\end{equation}
where $\delta_{\sigma_{\Lambda^c}}$ is the Dirac measure on
$(S^{\Lambda^c}, \mathcal{S}^{\Lambda^c})$, and the unique element
of $\mathcal{G}(\gamma^\chi)$ is the product measure $\chi^{\sf V}$.
Thus, one may expect that uniqueness persists if the dependence
encoded in $\gamma$ is in a sense weak. The typical procedure of
each work on the uniqueness of this kind, including the present one,
is to realize this idea in a given context.

Let $\rho=(\rho_\Lambda)_{\Lambda \in \mathcal{V}}$ be a family of
measurable functions $\rho_\Lambda : \Sigma \to \mathds{R}_{+} :=
[0, +\infty)$ such that $\rho\chi= (\rho_\Lambda
\chi_\Lambda)_{\Lambda \in \mathcal{V}}$ is a specification. Here
$\chi_\Lambda$ is as in (\ref{G2}). Then $\rho$ is called a
$\chi$-modification \cite[page 18]{Ge}. Furthermore, let
$(h_\Lambda)_{\Lambda \in \mathcal{V}}$ be a family of measurable
functions $h_\Lambda : \Sigma \to \mathds{R}_{+} := [0, +\infty)$
which enjoy the following consistency property
\begin{equation}
  \label{G3}
  h_\Lambda (\sigma) h_{\Lambda'} (\sigma') = h_\Lambda (\sigma') h_{\Lambda'}
  (\sigma),
\end{equation}
holding for all $\varnothing \neq \Lambda \subset \Lambda'$,
$\Lambda' \in \mathcal{V}$ and $\sigma, \sigma'\in \Sigma$ such that
$\sigma_{\Lambda^c} = \sigma'_{\Lambda^c}$. Assume also that
\begin{equation}
  \label{G4}
0< \chi_\Lambda (h_\Lambda) < \infty, \qquad  \chi_\Lambda
(h_\Lambda):=  \int_\Sigma h_\Lambda d \chi_\Lambda ,
\end{equation}
holding for all $\Lambda\in \mathcal{V}$. Then $h$ is called a
\emph{pre-modification}. In this case,  $\rho=(h_\Lambda/
\chi_\Lambda(h_\Lambda))_{\Lambda \in \mathcal{V}}$ is a
$\chi$-modification, see \cite[Remark 1.32, page 22]{Ge}.

If the pre-modification $h$ is such that
\begin{equation}
\label{HL} h_\Lambda (\sigma) = \exp\left( \Phi_\Lambda
(\sigma)\right),
\end{equation}
for a certain family of functions $\{\Phi_\Lambda\}_{\Lambda
\in\mathcal{V}}$, the corresponding random field is called a
\emph{Gibbs field}. They determine the \emph{model} one deals with.
In applications related to statistical physics, such functions are
called \emph{interaction potentials}. In the setting of this work,
necessary and sufficient conditions under which the elements of a
given pre-modification $h$ can be written as in (\ref{HL}) were
first obtained in \cite{Kozlov}. Thus, random fields satisfying
these conditions are Gibbs fields. Often, see e.g., \cite[page
9]{GeM} or \cite{H}, interaction potentials have the form
\begin{equation}
  \label{G5}
  \Phi_\Lambda (\sigma) = - \sum_{\{x,y\} \in \mathcal{V}^0_\Lambda} \varphi_{x,y} (\sigma_x, \sigma_y)
  ,
\end{equation}
for suitable symmetric measurable functions $\varphi_{x,y}: S^2 \to
\mathds{R}$, called \emph{binary (or two-body) interaction
potentials}. Here $\mathcal{V}^0$ is a collection of pairs $\{x,y\}$
and $\mathcal{V}_\Lambda^0=\{\{x,y\} \in \mathcal{V}^0: \{x,y\}\cap
\Lambda \neq \varnothing\}$. If one sets $x\sim y$ whenever $\{x,y\}
\in \mathcal{V}^0$, then the collection $\mathcal{V}^0$ determines a
simple graph  -- the \emph{dependence graph} for $h$ (and hence for
the corresponding specification) -- with vertex set $\sf V$ and the
adjacency relation as just mentioned. This dependence graph
determines the properties of the corresponding random fields. For
instance, Markov random fields with trees as dependence graphs have
many specific properties related to this particular feature, see
\cite[Chapter 12]{Ge}. The aforementioned possible `weakness' should
correspond to particular properties of the dependence graph and to
sufficiently small values of the logarithmic oscillations of
$h_\Lambda$, which in the case of (\ref{HL}) corresponds to small
oscillations of the interaction potentials.  For the Ising spin
model ($\sigma_x =\pm 1$ and binary interactions $\varphi_{x,y}
(\sigma_x, \sigma_y) = - a \sigma_x\sigma_y$) on a particular rooted
tree with vertices the degrees of which rapidly grow with distance
to the root, the corresponding Gibbs fields are multiple for any
$a>0$, see \cite{KK}. For a number of models with binary
interactions as in (\ref{G5}) and \emph{finite} single-spin spaces
$S$, a comprehensive analysis of the relationship between
uniqueness/nonuniqueness and the graph structure can be found in
\cite{GeM,H}. We also refer the reader to publications
\cite{CM,GR,Kozlov} for a detailed study of various interconnections
between Markov and Gibbs properties of random fields of this kind.

In several applications, modeling dependence by binary interactions
as in (\ref{G5}) proved insufficient and using higher-order
(many-body) interactions is being suggested, see e.g.,
\cite{Baxter,Can,Gandolfi,GandolfiL,Monroe,Tj}. In view of this, in
the present work we turn to the case where the dependence graph has
edges consisting of more than two elements of $\sf V$, i.e.,  is a
\emph{hypergraph}. Suppose that ${\sf V}$ and $\mathcal{V}$ are as
above, and there is given an infinite collection ${\sf E}\subset
\mathcal{V}$ of distinct subsets $e\subset {\sf V}$, none of which
is a singleton.
\begin{assumption}
  \label{1ass}
There is given a collection $(h_e)_{e\in {\sf E}}$ of measurable
functions $h_e:\Sigma\to \mathds{R}_{+}$ such that each $h_e$ is
$\mathcal{F}_e$-measurable and the following holds
\begin{equation}
  \label{G7}
 \forall e \in {\sf E} \ \  \forall \sigma \in \Sigma \qquad m_e
 \leq h_e (\sigma ) \leq M_e,
\end{equation}
for some $m_e, M_e$ such that $m_e >0$.
\end{assumption}
For $\Lambda\in \mathcal{V}$, set ${\sf E}_\Lambda= \{ e \in {\sf
E}: e\cap \Lambda \neq \varnothing\}$. Then define
\begin{equation}
  \label{G6}
h_\Lambda (\sigma) = \prod_{e\in {\sf E}_\Lambda} h_e(\sigma).
\end{equation}
It is clear that the family $(h_\Lambda)_{\Lambda \in \mathcal{V}}$
has the consistency property as in (\ref{G3}) and hence is a
pre-modification. Thus, for $\chi$ as above, we have that
$\chi_\Lambda (h_\Lambda)
>0$, see (\ref{G4}), and hence $\rho =(h_\Lambda / \chi_\Lambda
(h_\Lambda))_{\Lambda \in \mathcal{V}}$ is a $\chi$-modification.
Thereby, $\gamma = \rho \chi = (\rho_\Lambda \chi_\Lambda)_{\Lambda
\in \mathcal{V}}$ is a specification. Our aim is to establish a
sufficient condition imposed on the collections $\sf E$ and
$(m_e)_{e\in {\sf E}}$, $(M_e)_{e\in {\sf E}}$, under which
$\mathcal{G}(\gamma)$ is a singleton. The fact that
$\mathcal{G}(\gamma) \neq \varnothing$ follows by our assumption
that $(S, \mathcal{S})$ is a standard Borel space, see \cite[Theorem
8.7, page 142]{Ge}. For $\Lambda \in \mathcal{V}$, we  set
\begin{equation*}
  {\sf E}^o_\Lambda = \{e \in {\sf E}: e\subset \Lambda\}, \qquad
  \partial {\sf E}_\Lambda = \{ e\in {\sf E}_\Lambda: e\cap
  \Lambda^c \neq \varnothing\},
\end{equation*}
and also
\begin{equation}
  \label{Ga8a}
\partial \Lambda = \bigcup_{e \in \partial {\sf E}_\Lambda} \left(
e \cap \Lambda^c\right).
\end{equation}
Since all $e\in {\sf E}$ are finite, ${\sf E}^o_\Lambda \neq
\varnothing$ for sufficiently big $\Lambda$'s. By (\ref{G6}) it
directly follows that the elements of $\mathcal{G}(\gamma)$ are
\emph{locally} Markov random fields in the sense that
\begin{equation}
  \label{Ga9}
  \mu(A|\mathcal{F}_{\Lambda^c}) =  \mu(A|\mathcal{F}_{\partial
  \Lambda}), \qquad A\in \mathcal{F}
\end{equation}
holding for all $\Lambda \in \mathcal{V}$. Following
\cite{AHK,AHKO,F,Israel}, we say that $\mu\in \mathcal{G}(\gamma)$
is \emph{globally} Markov, if (\ref{Ga9}) holds for all $\Lambda
\subset {\sf V}$. The significance of this property is discussed,
e.g., in \cite{AHK,b2,G,Israel}. For finite single-spin spaces $S$,
the existence of the global Markov property may be related to number
of properties of the corresponding random fields, see, e.g.,
\cite{Coq,G} and also \cite{Israel} where its absence is shown and
discussed. One of our aims is to relate this property to the
uniqueness condition which we are going to derive.

Among the most known tools of proving uniqueness there is the
celebrated Dobrushin criterion, see \cite[Chapter 8]{Ge}. In the
present context, its condition is satisfied if the following holds,
see \cite[Proposition 8.8, page 143]{Ge},
\begin{gather}
  \label{G8}
  \sup_{x\in {\sf V}} \sum_{e\in {\sf E}_x} \left( |e| - 1 \right)
  \delta(e) < 2,
\end{gather}
where
\begin{gather}
\label{G8a}
    {\sf E}_x := \{ e\in {\sf E}: x\in e\},
  \qquad \delta (e) := \log M_e - \log m_e,
\end{gather}
and $|e|$ standing for the cardinality of $e\subset {\sf V}$. If one
assumes uniform boundedness of $\delta(e)$, then the condition in
(\ref{G8}) can be satisfied only in the rather trivial case of
uniform boundedness of both $|e|$ and $|{\sf E}_x|$. In the present
work, instead of that in (\ref{G8}) we obtain (in Theorem \ref{1tm})
a condition (see (\ref{Gb2b}) below) which works also for unbounded
$|e|$, $|{\sf E}_x|$ and $\delta(e)$. In \cite{AHK,AHKO,F}, it was
shown that Dobrushin's uniqueness implies that the unique $\mu\in
\mathcal{G}(\gamma)$ is globally Markov. We prove that a similar
statement is true also in our case: the fulfillment of the
uniqueness condition (\ref{Gb2b}) implies the global Markov property
of the unique $\mu\in \mathcal{G}(\gamma)$.

The rest of the paper has the following structure. In Section
\ref{US}, we introduce all mathematical necessities and then
formulate our main result in Theorem \ref{1tm}. Thereafter, we
provide a number of comments aimed at indicating the significance of
the present result and its role in the context of the theory of
Markov fields of this kind. The same aim is pursued in Section
\ref{S} where we provide two examples. In the first one, we
introduce a model, to a version of which the classical Dobrushin
criterion is not applicable as $|e|$ are unbounded in this case. At
the same time, Theorem \ref{1tm} does work for this model, yielding
a result obtained therein. In the second example, we deal with
random interactions and obtain a result -- Theorem \ref{2tm} -- that
can be considered as an extension of our previous result in
\cite{KK2} to the case of higher-order dependencies. Finally,
Section 4 contains the proof of both mentioned theorems.

\section{The Result}
\label{US}
 As mentioned above, the structure of the dependence graph
corresponding to a given specification predetermines the properties
of $\mathcal{G}(\gamma)$. Thus, we begin by introducing a special
kind of graphs and some related facts  and terminology based on our
previous works \cite{KK1,KK2}. Recall that, for a finite set
$\Lambda$, by $|\Lambda|$ we denote its cardinality. For a subset,
e.g., $\Delta \subset {\sf V}$, by $\Delta^c$ we denote the
complement, i.e., ${\sf V}\setminus \Delta$.

\subsection{Graphs of tempered degree growth}

Let $ G=( V, W)$ be a countably infinite simple graph, i.e., it does
not have loops and multiple edges. We shall also assume that $G$
does not have isolated vertices. Each edge $w \in W$ is identified
with a pair $x,y \in X$, and we write $x\sim y$ if there exists
$w\in W$ such that $w=\{x,y\}$. A path $\vartheta (x,y)$ in $G$ is a
sequence $x_0 , x_1, \dots , x_n$ such that $x_0 = x$, $x_n =y$ and
$x_l \sim x_{l+1}$ for all $l=0, \dots , n-1$. Its length
$|\vartheta(x,y)|$ is set to be $n$. A path is called \emph{simple}
if all $x_l$, $l=0, \dots , n$ are distinct. Then $d_G (x,y) =
\min_{\vartheta(x,y)} |\vartheta (x,y)|$ is a metric on $V$ if one
sets $d_G (x,y) = +\infty$ whenever there is no $\vartheta$
connecting $x$ and $y$. A subgraph, $G'\subset G$, is a graph whose
vertex set, $V'$, and edge set, $W'$, are subsets of $V$ and $W$,
correspondingly. We say that a subset $V'\subset V$ generates a
subgraph $G'$ if the edge set of the latter consists of all those
$w\in W$ which satisfy $w\subset V'$. For a subgraph, $G'$, by
$d_{G'}$ we mean the metric defined as above with the use of
vertices and edges of $G'$ only. It is clear that $d_G (x,y) \leq
d_{G'} (x,y)$ in that case. A graph $G$ (resp. subgraph $G'$) is
said to be connected if $d_G(x,y)<\infty$ (resp. $d_{G'}(x,y)<
\infty$) for all its vertices $x$ and $y$. For $x\in V$, we set $W_x
=\{ w\in W: x\in w\}$. Then the degree of $x$, denoted $n_G (x)$, is
set to be the cardinality of $W_x$. Recall that we assume $n_G(x)>0$
for each $x$. Additionally, we assume that $n_G (x) < \infty$ for
all $x\in V$, i.e. the graph $G$ is \emph{locally finite}. If
\begin{equation}
  \label{Ga1}
  \sup_{x\in V} n_G(x) =: \bar{n}_G < \infty,
\end{equation}
the graph $G$ is said to be of bounded degree. In general, we do not
assume (\ref{Ga1}) to hold for our graphs.

Let $A$ be a finite nonempty subset of $V$, and let $A$  denote also
the subgraph generated by $A$. We call $A$ an \emph{animal} if it is
connected. Let $\vartheta$ be a simple path; by $A_\vartheta$ we
denote the graph generated by the vertices of $\vartheta$. Clearly,
$A_\vartheta$ is an animal. Let $g:\mathds{N}\to \mathds{R}_{+}$ be
a strictly increasing function. For an animal, $A$, set
\begin{equation}
  \label{Ga2}
  \mathfrak{G}(A; g) = \frac{1}{|A|} \sum_{x\in A} g(n_G(x)).
\end{equation}
If the graph is of bounded degree, then $\mathfrak{G}(A;g) \leq
g(\bar{n}_G)$ for all subsets $A\subset V$. On the other hand, if
the graph fails to satisfy (\ref{Ga1}), then $\mathfrak{G}(A;g)$ can
be made arbitrarily big by taking $x$ with big enough $n_G(x)$ (a
hub) and small enough $A$. These observations point to the
possibility of controlling the increase of $n_G(x)$ by means of the
averages as in (\ref{Ga2}). That is, one can say that $G$ has a
tempered degree growth if $\mathfrak{G}(A; g)$ is globally bounded
for specially selected animals of arbitrarily big cardinalities.
This, in particular, means that hubs ought to be sparse in such a
graph, see \cite{KK1} for further explanations. To make this idea
more precise, we take $B_r(x) = \{ y\in V: d_{G} (x,y) \leq r\}$,
$r>0$, $x\in V$, and then define
\begin{equation}
  \label{Ga3}
  \mathcal{A}_r(x) = \{ A \subset B_r(x): x\in A, \ |A| \geq r+1 \ {\rm and} \ A \ {\rm is} \ {\rm an} \
  { \rm animal} \}.
\end{equation}
\begin{definition}
  \label{Ga1df}
For given $g$ and $\bar{a}>0$, let $\mathfrak{G}(A,g)$ be as in
(\ref{Ga2}). The graph $G$ is said to be $(g,\bar{a})$-tempered if
for each $x\in V$, there exists a strictly increasing sequence
$\{N_k\}_{k\in \mathds{N}}\subset \mathds{N}$ such that the
following holds
\begin{equation*}
 \sup_{x\in V}\sup_{k\in \mathds{N}} \max_{A\in \mathcal{A}_{N_k} (x)}
\mathfrak{G}(A;g) = \bar{a} .
\end{equation*}
\end{definition}
The property just introduced is closely related to the properties of
simple paths in the corresponding graph. In view of this, we
introduce the following families of them.
\begin{definition}
  \label{BBdf}
By $\varTheta_r(x)$ we denote the family of all simple paths
$\vartheta$ originated at $x\in V$ such that $A_\vartheta \in
\mathcal{A}_r(x)$. That is, $A_\vartheta\subset B_r(x)$ and
$|\vartheta|\geq r$. Furthermore, for $N\geq r+1$, we  set
\begin{equation}
  \label{Ga33}
  \varTheta^N_r(x) = \{ \vartheta\in \varTheta_r(x): |\vartheta|=N -1\}.
\end{equation}
Note that $|A_\vartheta|=N$ whenever $\vartheta \in
\varTheta^N_r(x)$.
\end{definition}
 The advantage of dealing with tempered graphs
can be seen from the following fact, proved in our previous work.
\begin{proposition}
  \label{Ga1pn} \cite[Proposition 4]{KK2}
Let $G$ be $(g,\bar{a})$-tempered for $g(n) = \log n$ and some
$\bar{a}$.  For a given $x\in V$, let $\{N_k\}$ be as in Definition
\ref{Ga1df}. Then
\begin{equation}
  \label{Ga5}
\forall k \ \forall N\geq N_k  \qquad |\varTheta_{N_k}^N (x) |\leq
\exp( \bar{a} N).
\end{equation}
\end{proposition}
The following family of graphs has the property as in Definition
\ref{Ga1df}. For $x,y\in V$, set  $m_{-} (x,y)=\min\{n_G (x);
n_G(y)\}$. Let $\phi: \mathds{N}_0 \to \mathds{R}_{+}$ be an
increasing function. A graph $G$ is said to belong to the collection
$\mathds{G}_{-}(\phi)$ if it satisfies
\begin{equation}
  \label{Ga6}
d_G (x,y) \geq \phi (m_{-} (x,y)),
\end{equation}
whenever $m_{-} (x,y) \geq n_*$ for a $G$-specific $n_*>0$. By
virtue of (\ref{Ga6}), in such graphs hubs should be sparse. In a
slightly different form, objects of this kind appeared in \cite{BD}
as a tool of proving uniqueness for fields with random binary
interactions, i.e., where $h_\Lambda$ are as in (\ref{G5}) with
random $\varphi_{x,y}$, see also \cite{KK,KK2} for more on this
subject. Their temperedness is established by the following
statement.
\begin{proposition}
  \label{Ga2pn} \cite[Proposition 1]{KK2}
Assume that there exists a strictly increasing sequence $\{t_k
\}_{k\in \mathds{N}}\subset \mathds{N}$, $t_k \to  +\infty$, such
that $g$ and $\phi$ satisfy
\begin{equation}
  \label{Ga7}
  \sum_{k=1}^\infty \frac{g(t_{k+1})}{\phi (t_k)} =: \bar{b} <\infty.
\end{equation}
Then each $G\in \mathds{G}_{-}(\phi)$ is $(g,2\bar{b})$-tempered.
\end{proposition}
Note that the rooted tree used in \cite{KK}, for which Gibbs states
of an Ising model are multiple for all nonzero interactions, is not
tempered as the degree of a given $x$ is $l!$, where $l$ is its
distance to the root. Then the neighbor of this $x$ located at
distance $l+1$ from the root has degree $(l+1)!$, i.e., is even
bigger hub than $x$.
\begin{remark}
  \label{G1rk}
Let $G$ be $(g,\bar{a})$-tempered with some $\bar{a}>0$ and $g(n)
\geq g_* (n) := \log n$. As follows from the proof of
\cite[Proposition 4]{KK2}, the cardinality of the set defined in
(\ref{Ga33}) satisfies
\begin{gather*}
|\varTheta_{N_k}^N (x) |\leq \max_{\vartheta \in \varTheta_{N_k}^N
(x)} \exp\left(N \mathfrak{G}(A_\vartheta;g_*)  \right)\leq
\max_{\vartheta \in \varTheta_{N_k}^N
(x)} \exp\left(N \mathfrak{G}(A_\vartheta;g)\right)\\[.2cm] \nonumber \leq \max_{A \in \mathcal{A}_{N_k}^N
(x)} \exp\left(N \mathfrak{G}(A;g) \right) \leq \exp\left(\bar{a} N
\right).
\end{gather*}
That is, $g_*$ is the smallest known $g$ such that the $g$-tempered
graphs admit estimates as in (\ref{Ga5}).
\end{remark}

\subsection{Formulating the result}

As already mentioned, the specification $\gamma$ we are going to
deal with is defined by a family $(h_e)_{e\in {\sf E}}$ satisfying
Assumption \ref{1ass}. Let us turn now to the hypergraph structure
of the underlying set ${\sf V}$ defined by $\sf E$. Here we mostly
employ the terminology of \cite{HypG}. In this context, each $e\in
{\sf E}$ serves as a hyperedge, `connecting' its elements $x,y,z,
\dots \in e$. We denote this dependence hypergraph by ${\sf H}$; its
vertex and edge sets are ${\sf V}$ and $\sf E$, respectively. For
$\Delta \subset {\sf V}$, we then set
\begin{equation}
  \label{Gb0}
  \partial \Delta = \{ y\in \Delta^c: \exists x \in \Delta \ \exists
  e \in {\sf E} \ \{x,y\} \subset e\},
\end{equation}
which is just an extension of (\ref{Ga8a}) to infinite subsets. For
$x\in {\sf V}$, we set
\begin{equation}
  \label{Gb}
 n_{\sf H} (x) = |{\sf E}_x
  |,
\end{equation}
where ${\sf E}_x$ is as in (\ref{G8a}). Note that ${\sf E}_x$ and
$n_{\sf H} (x)$ are the edge neighborhood and the edge degree of
$x$, respectively. We will assume that $\sf H$ is locally finite,
i.e.,
\begin{equation*}
  \forall x\in {\sf V} \qquad n_{\sf H} (x) < \infty.
\end{equation*}
A convenient way of describing hypergraphs is to use their
line-graphs, see \cite[Sect. 2.1]{HypG}. For a given hypergraph $\sf
H$, its line-graph ${\sf L}({\sf H})$ has $\sf E$ as the vertex set,
and $e, e' \in {\sf E}$ are declared incident if $e\cap e' \neq
\varnothing$. For a given $e\in {\sf E}$, we then set
\begin{equation}
  \label{Gb2}
  \mathcal{E}_e = \{ e'\in {\sf E}: e'\sim e\}, \qquad n_{\sf L}(e)
  = |\mathcal{E}_e|,
\end{equation}
which is the vertex neighborhood and the vertex degree of $e$ in
${\sf L}({\sf H})$, respectively. Let $\vartheta$ be a simple path
in ${\sf L}({\sf H})$ and ${\sf A}_\vartheta$ the corresponding
animal. Set, cf. (\ref{Ga2}),
\begin{equation}
  \label{Gb2a}
\mathfrak{D}(\vartheta) = \frac{1}{|{\sf A}_\vartheta|} \sum_{e\in
{\sf A}_\vartheta} \log \left(e^{2\delta(e)}-1 \right),
\end{equation}
where $\delta (e)$ is the logarithmic oscillation of $h_e$ defined
in (\ref{G8a}). For $e\in {\sf E}$, by $\mathcal{A}_r(e)$ and
$\varTheta_r(e)$ we denote the families of animals and paths in
${\sf L}({\sf H})$, respectively, defined according to (\ref{Ga3})
and Definition \ref{BBdf}. Then the temperedness of ${\sf L}({\sf
H})$ is set according to Definition \ref{Ga1df} with the use of
$\mathcal{A}_r(e)$. Now we are at a position to state our main
result.
\begin{theorem}
  \label{1tm}
Let the pre-modification $h$, and thus the specification $\gamma$,
be such that the line-graph ${\sf L}({\sf H})$ be
$(g,\bar{a})$-tempered with $g(n) \geq g_* (n) = \log n$ and a
certain $\bar{a}>0$. Then $\mathcal{G}(\gamma)$ is a singleton
whenever,
\begin{equation}
  \label{Gb2b}
\sup_{e\in {\sf E}}  \sup_{k\in \mathds{N}}\max_{\vartheta\in
\varTheta_{N_k}(e)}\mathfrak{D}({\vartheta}) \leq -( \bar{a} +
\epsilon),
\end{equation}
holding for some $\epsilon >0$. In this case, the unique $\mu\in
\mathcal{G}(\gamma)$ is globally Markov, i.e., such that
$\mu(\cdot|\mathcal{F}_{\Delta^c} ) =
\mu(\cdot|\mathcal{F}_{\partial \Delta} )$ $\mu$-almost surely,
holding for all $\Delta \subset {\sf V}$, see (\ref{Gb0}).
\end{theorem}
The strength of the result just stated is illustrated in Section
\ref{S} below where we obtain uniqueness for a model to which
Dobrushin's method based on (\ref{G8}) is not applicable. The proof
of Theorem  \ref{1tm} follows in Section \ref{SS}. Let us turn now
to discussing it. The essential aspects of this statement are:
\begin{itemize}
  \item[(i)] According to the uniqueness condition in (\ref{Gb2b}), the
  dependencies encoded in $\gamma$
 should be weak `in average'. This makes it applicable also in cases
 where the dependencies are `irregular' or random.
\item[(ii)] The just mentioned irregularity may characterize both the dependence graph and the interaction strength.
That is, neither $n_{\sf L}(e)$ nor $|e|$ are supposed to be
  bounded, i.e., satisfying the corresponding versions of
  (\ref{Ga1}). We allow $\delta(e)$ to be unbounded as well.
 \item[(iii)] The condition in (\ref{Gb2b}) is imposed on families  $\varTheta_r(e)$ of simple
 paths in ${\sf L}({\sf H})$ of lengths satisfying a certain $e$-dependent condition -- not for all $r$.
 According to it, each $\vartheta \in \varTheta_r(e)$
 may have $e'$ with an arbitrarily big
  $\delta(e')$,
  `diluted' by $e''$, for which $e^{2\delta(e'')} - 1 < 1$.
\end{itemize}
The  Dobrushin uniqueness criterion was formulated in \cite{D} as a
condition of weak dependence of $\gamma_{\{x\}}(\cdot|\sigma)$ on
$\sigma_y$, $y\sim x$. In \cite{DS}, this condition was called
$C_{x}$ condition, and the approach was extended by formulating
analogous conditions $C_\Lambda$ for finite $\Lambda$ covering the
underlying set. Soon after, Dobrushin's approach became a
cornerstone in the theory of Gibbs random fields, see, e.g., the
bibliography in \cite{DS}. In particular, under the condition of a
complete analyticity of the interaction potentials, see
\cite{DS1a,DS1,Mart}, uniqueness was shown to hold. It should be
noted, however, that completely analytic potentials should possess
some uniformity, e.g., be translation invariant. Moreover, in the
mentioned works the single-spin space $S$ is finite, (see also
\cite{Israel0} where $S$ is just compact), which means that the
arguments used in \cite{DS1a,DS}, or in \cite{Israel0,Mart}, should
be essentially modified to be applicable in the present context
where $S$ is just a standard Borel space.

For models with finite $S$, there has been elaborated a more
efficient technique of proving uniqueness than that based on
Dobrushin's condition, see \cite{vdB,vdBM} and also \cite[Chapter
7]{GeM}. In this technique, uniqueness of a Gibbs random field on a
given graph is obtained by showing absence of the bond Bernoulli
(disagreement) percolation on the same graph, see \cite[Theorem
1]{vdBM} or \cite[Theorem 7.2, page 76]{GeM}. In our case, the key
ingredient of the theory is the estimate of the number of paths
given in Proposition \ref{Ga1pn}, see also Remark \ref{G1rk}.
According to (\ref{Ga5}), the percolation threshold $p_c$ for the
Bernoulli site percolation is just $e^{-\bar{a}}$, which can readily
be obtained as in (\ref{Gc19}) below. That is, the main aspect of
our approach is that we work directly with the underlying graph, the
properties of which give rise to both non-percolation and uniqueness
of Markov random fields.

As might follow from the condition in (\ref{G8}), the original
Dobrushin criterion is universal as it can be used to random fields
with two-body or multi-body interactions, which may also have long
range. It is for a special type of such interactions, this criterion
is sharp, see \cite[Theorem 16.27]{Ge}.  We hope that, for binary
interactions with finite range, our approach can yield more refined
uniqueness conditions if one takes  $e$ selected in an `optimal'
way.
 Another direction where an appropriate modification of our approach
can be of use is studying random fields corresponding to
hierarchically arranged sets $\sf E$, see, e.g., \cite{BM}. We plan
to turn to these issues in a forthcoming work.

\subsection{Further facts and comments}

The following statement provides a more explicit version of the
uniqueness condition in (\ref{Gb2b}).
\begin{proposition}
  \label{Gb1pn}
Let $\sf H$, $g$ and $\bar{a}$ be as in Theorem \ref{1tm}.
 Assume that, for some $\epsilon\in (0,1)$, there exists $\varkappa \in
 (0,1)$, dependent on  $\bar{a}$ and $\epsilon$ only,  such that the
 following holds
 \begin{equation}
   \label{H}
\forall e\in {\sf E} \qquad \delta (e) \leq \varkappa [\bar{a} + g(
n_{\sf L}(e))].
 \end{equation}
Then (\ref{Gb2b}) is satisfied with this $\epsilon$.
\end{proposition}
\begin{proof}
By (\ref{H}) we have
\begin{gather}
  \label{H1}
  e^{2\delta (e)} -1 \leq 2\delta(e) e^{2\delta (e)}
   \leq \exp \bigg{(} \log \varkappa + (2 \varkappa+1) [\bar{a} + g( n_{\sf L}(e))]  \bigg{)} \\[.2cm] \nonumber \leq
   \exp \bigg{(} \log \varkappa + 3 [\bar{a} + g( n_{\sf L}(e))]  \bigg{)}.
\end{gather}
We use this and (\ref{H1}) in (\ref{Gb2a}) and arrive at
\[
\sup_{e\in {\sf E}} \sup_{k\in \mathds{N}}\max_{\vartheta\in
\varTheta_{N_k}(e)}\mathfrak{D} (\vartheta) \leq \log \varkappa + 6
\bar{a},
\]
which means that (\ref{Gb2b}) holds for $\varkappa = e^{- 7 \bar{a}
- \epsilon}$, which yields the proof.
\end{proof}

To clarify the interconnections between the properties of ${\sf H}$,
${\sf L}({\sf H})$ and $(h_e)_{e\in {\sf E}}$, let us impose some
further restrictions. In particular, we assume that
\begin{equation}
  \label{Gb3}
  \bigcap_{e\in {\sf E}_x} e = \{x\}.
\end{equation}
This separability assumption implies the following.
\begin{itemize}
  \item[(a)] Each $e$ can contain at most one $x$ such that $n_{\sf
  H}(x) =1$, i.e., such that is contained only in this $e$.
  \item[(b)] If distinct $x$ and $y$ are contained in a given $e$,
  and if they are contained in a certain $e'\sim e$, then there
  exist another $e_1, e_2 \in {\sf E}$ such that $x\in e_1$, $y\in
  e_2$, and also $x\in e_2^c$ and $y\in e_1^c$.
\end{itemize}
These observations further yield
\begin{equation}
  \label{Gb4}
  |e|- 1 \leq n_{\sf L}(e) \leq \sum_{x\in e} (n_{\sf H}(x) -1) \leq
  |e| \max_{x\in e} (n_{\sf H}(x) -1).
\end{equation}
The proof of (a) is obvious. To prove (b) one observes that the
assumption that $y$ is out of any $e_1$ which contains $x$
contradicts (\ref{Gb3}). The lower bound for $n_{\sf L}(e)$ in
(\ref{Gb4}) follows first by (a) -- as each but possibly one of
$y\in e$ should have $e' \in \mathcal{E}_e$, see (\ref{Gb2}) -- and
then by (b), which yields that each such $y$ should have its own
$e'$. The upper bound follows by (b), where subtracting 1 is to take
into account that $x$ belongs to $e$. Clearly (\ref{Gb3}) is
essential for the properties of $\sf H$. However, from the point of
view of the application to the random fields considered here it is
not too restrictive for the following reason. First of all, one has
to mention that the construction can be modified to include the
possibility to take the single-state (spin) space $S$ dependent on
$x$, i.e., to take $\prod_{x\in {\sf V}} S_x$ instead of $S^{\sf
V}$. Then one repeats the whole construction after imposing suitable
conditions on $S_x$, uniform in $x$. Now if (\ref{Gb3}) fails to
hold, one introduces the equivalence relation and passes to
equivalence classes (each is finite) consisting of those $x, y,...$
that belong to the intersections. Then the factor-hypergraph
obtained thereby enjoys the separability in question. Thereafter,
one sets strings $\sigma_{[x]} = (\sigma_x, \sigma_y ,\dots)$ as new
spins lying in the corresponding product spaces and described by the
corresponding products of $\chi$. Similar arguments we used in
\cite{BD} where $x$ and $y$ were declared similar if $\varphi_{x,y}$
is big, which allowed for including such terms in the new reference
measure in place of $\chi^{\sf V}$.

The bounds obtained in (\ref{Gb4}) show that $n_{\sf L}(e)$ can be
big (and hence $e$ can be a hub) because: (a) $|e|$ is big; (b) many
of $x\in e$ have many neighbors in $\sf H$. The bounds as in
(\ref{Gb4}) can be used as follows. Assume that, cf. (\ref{Gb}),
\[
\sup_{x\in {\sf V}} n_{\sf H}(x) =: \bar{n}_{\sf H} < \infty.
\]
Then $n_{\sf L}(e)$ satisfies the following two-sided estimate
\begin{equation*}
  |e|- 1 \leq n_{\sf L}(e) \leq |e| (\bar{n}_{\sf H} -1),
\end{equation*}
and hence is controlled by $|e|$. This may allow one to use $|e|$ in
definitions like (\ref{Ga2}) and (\ref{Ga6}) instead of $n_{\sf
L}(e)$. Also, if $d_{\sf L}(e,e')$ satisfies
\begin{equation*}
  d_{\sf L}(e,e') \geq \phi(\min\{|e|; |e'|\}),
\end{equation*}
 for $\phi$ and $g=g_*$ satisfying (\ref{Ga7}), see Remark
 \ref{G1rk}, then ${\sf L}({\sf H})$ satisfies the conditions of
 Theorem \ref{1tm}. This might yield a more direct and thus convenient way of
 controlling the graph $\sf H$ than that based on the degrees
 $n_{\sf L}$.

\section{Examples}

\label{S}

To illustrate the abilities of the uniqueness condition formulated
in Theorem \ref{1tm}, we provide two examples. We begin by
introducing a class of models where the underlying line-graphs are
tempered.

\subsection{The overlapping cliques model}
Let $\{n_l\}_{l\in \mathds{N}} \subset \mathds{N}$ be given. With
the help of this sequence, we furnish $\sf V$ with the following
structure. One starts with a subset  $e_1 \subset {\sf V}$
containing $n_1=|e_1|$ elements. Then one takes subsets $e_{2,1},
\dots, e_{2,n_1}$ containing $n_2$ elements each, in such a way
that: (a) $e_{2,i} \cap e_{2,j} = \varnothing$, $i\neq j$; (b) $e_1
\cap e_{2,i}$ is a singleton, the unique element of which is denoted
$x_i$. In other words, each $x_i \in e_1$ is an element of the
corresponding $e_{2,i}$. Next, one takes $e_{3,i,k}$, $i=1, \dots
n_1$, $k=1, \dots ,n_2-1$. All $e_{3,i,k}$ are pairwise disjoint,
each
 containing $n_3$ elements and $e_{3,i,k} \cap e_{2,i} =
 \{x_{i,k}\}$. That is, each $x_{i,k} \in e_{2,i}$ is an element of
$e_{3,i,k}$. This procedure is then continued ad infinitum
 yielding $${\sf E}:=\{e_{m, i_1, \dots , i_{m-1}}: m\in \mathds{N}, i_1\leq n_1, i_2\leq n_2-1, \dots i_{m-1} \leq n_{m-1}
 -1\},$$ which exhausts $\sf V$. Thereby, one gets the hypergraph ${\sf H}= ({\sf V}, {\sf E})$. Its linear graph ${\sf L}({\sf H})$ is a rooted
 tree of which $\{n_l\}$ is the degree sequence.
If one sets an appropriate graph structure on each $e_{m, i_1, \dots
,
 i_{m-1}}$, then the resulting graph becomes a Husimi tree, see
 \cite[page 259]{Monroe}, \cite[page 220]{Monroe1}, or \cite{Essam}. In particular, one can
 turn each $e_{m, i_1, \dots ,
 i_{m-1}}$ into a clique - a complete graph $K_{n_m}$, see
 \cite{Essam}.
Now we define, cf.
 (\ref{G7}),
 \begin{equation}
 \label{Aa39}
h_e(\sigma) = \exp \left(K\varphi_e (\sigma_e) \right), \qquad 0<
a_e \leq \varphi_e(\sigma_e)\leq b_e, \quad K>0,
 \end{equation}
which yields
\begin{equation}
  \label{Aa1}
  \delta (e) =  K (b_e -a_e) =: K c_e.
\end{equation}
Recall that each $e_{m,i_1, \dots , i_{m-1}}$ with the same $m$
contains the same number $n_m$ of elements. Assume  that they are
isomorphic if furnished with a graph structure. This, in particular,
means that all $\varphi_e$, and hence $c_e$, are the same for such
$e$. In this case, we also write $e_m$ meaning one of these sets. An
example can be the model as in \cite{Monroe}, where each $n_l=4$ and
all $e$ are $C_4$ with the corresponding $\varphi_e$, see also
\cite[page 220]{Monroe1}, where $n=3$. By construction, each ${\sf
E}_x$, see (\ref{G8a}), contains two elements: $e_m$ and $e_{m+1}$
for $x$-dependent $m$. For instance, each $x\in e_1$ is contained
also in some $e_2$. Each but one element of $e_{2,i}$ is contained
also in $e_{3,i,k}$ for some $k$. In this case, the Dobrushin
uniqueness condition (\ref{G8}) takes the form
\begin{equation}
  \label{Aa2}
K  \sup_{m\geq 1} \left[ c_{e_m} (n_m-1) + c_{e_{m+1}}
  (n_{m+1}-1)\right] < 2.
\end{equation}
Thus, to verify (\ref{Aa2}) both sequences $\{c_{e_m}\}$ and
$\{n_m\}$ ought to be bounded. If they all are constant, i.e.,
$c_{e_m}=c$ and $n_m=n$,  as it is in \cite{Monroe,Monroe1}, then
(\ref{Aa2}) turns into
\begin{equation}
  \label{Aa23}
 K < 1/ c (n -1).
\end{equation}
In this case, the line graph ${\sf L}({\sf H})$ is an $n$-tree and
we have that $n_{\sf L}=n$; hence, $\bar{a} = \log n$, see
Definition \ref{Ga1df}. And also $\mathfrak{D}(\vartheta)
=\mathfrak{D}(e) = \log(e^{2Kc} -1)$, which means that the condition
in (\ref{Gb2b}) takes the form
\begin{equation*}
  K < \frac{1}{2c} \log \frac{n+1}{n},
\end{equation*}
that is comparable with (\ref{Aa23}) and has a similar asymptotic as
$n\to +\infty$. For unbounded sequences $\{c_{e_m}\}$ and $\{n_m\}$
-- where (\ref{G8}), hence (\ref{Aa2}), does not work -- we have the
following result. Assume that there is given an increasing sequence
$\{l_s\}\subset \mathds{N}$, which we use to impose the following
conditions on the sequence $\{n_l\}$: $n_{1+s}\leq n_1$, for $1\leq
s\leq l_1-1$; $n_{1+l_1+ s} \leq n_{1+l_1}$ for $1\leq s\leq l_2-1$,
an so on. In other words, the increase of $n_m$ is allowed only when
$m$ takes values $m_s$, $s\geq 1$; where $m_1=1$ and $m_{s}:=1 + l_1
+\cdots + l_{s-1}$ for $s\geq 2$. Moreover, the elements of these
sequences are supposed to satisfy
\begin{equation}
  \label{Aa25}
l_s \geq \phi(n_{m_s}) := \left[ \log n_{m_s} \right]^2.
\end{equation}
That is, the tree -- the line-graph of the hypergraph defined above,
satisfies (\ref{Ga6}) with this $\phi$. Now we take $t_k = \exp(
a^k)$, $k\in \mathds{N}$ for some $a>1$. For this sequence and $g(t)
= \log t$, we then have
\begin{gather}
  \label{Aa26}
  \bar{b}= \sum_{k=1}^\infty \frac{g(t_{k+})}{\phi(t_k)} =
  a\sum_{k=1}^\infty a^{-k} = \frac{a}{a-1}.
\end{gather}
By Proposition \ref{Ga2pn} it follows that the tree ${\sf L}({\sf
H})$ is $(g,\bar{a})$-tempered with such $g$ and $\bar{a} =
2a/(a-1)$. By Proposition \ref{Gb1pn} we then conclude that the set
of the corresponding Markov fields is a singleton if
\begin{equation}
  \label{Aa27}
  K c_{e_m} < e^{-14}( 2+ \log n_m),
\end{equation}
by which $c_{e_m}$ can also be unbounded. Assume now that
$\varphi_e(\sigma_e)$ has the following form
\begin{equation}
  \label{Aa28}
  \varphi_e(\sigma_e) = \frac{1}{|e|} \sum_{\{i,j\} \subset e}
  \sigma_i \sigma_j , \qquad \sigma_i =\pm 1,
\end{equation}
which means that each $e$ is turned into a clique, and the
interaction within each clique is of Curie-Weiss type. That is, we
are dealing with an Ising spin model on a Husimi tree of this kind.
Simple calculations mean that
\begin{equation}
  \label{Aa29}
  c_e = \frac{|e|-1}{2} +\frac{1}{|e|} \left[ \frac{|e|}{2}\right]
  \leq \frac{|e|}{2},
\end{equation}
where $[\alpha]$ stands for the integer part of $\alpha>0$. That is,
$c_e$ as in (\ref{Aa29}) fails to satisfy (\ref{Aa27}) for big $|e|$
if the growth of the sequence $\{n_l\}$ is governed by (\ref{Aa25}).
In view of this, we impose a more restrictive condition on the
growth of this sequence. Namely, instead of (\ref{Aa25}) we set
\begin{equation*}
  l_s \geq \phi(n_{m_s}) = [n_{m_s}]^2.
\end{equation*}
In this case, for $t_k = a^{k}$, $a>1$, we have that (\ref{Aa26})
holds true for $g(t)=t$ and $\phi (t) = t^2$. Thus, ${\sf L}({\sf
H})$ is $(g,\bar{a})$-tempered with this $g$ and $\bar{a}=2a/(a-1)$.
Then the conditions of Theorem \ref{1tm} are satisfied if, cf.
(\ref{Aa27}) and (\ref{Aa29}),
\begin{equation*}
K  < 2 e^{-14} ,
\end{equation*}
where we have taken into account that $n_m\geq 2$.

\subsection{Random interactions}

In this subsection, we obtain a generalization to higher-order
dependencies (many-body interactions) of our previous result in
\cite{KK2}, which in turn is a refinement and a generalization of
the classical Bassalygo-Dobrushin work \cite{BD}. Assume that the
functions $h_e$ -- which determine the specification $\gamma$, and
hence the set $\mathcal{G}\gamma)$ -- are random. Namely, each $h_e$
is as in (\ref{Aa39}) with a random function $\varphi_e$, for which
the bounds $a_e$ and $b_e$ are also random. In particular,
$\varphi_e$ may be as in (\ref{Aa28}) times a random numerical
factor $q_e$. In this case, the set $\mathcal{G}(\gamma)$ also gets
random.
\begin{theorem}
  \label{2tm}
Let the line-graph ${\sf L}({\sf H})$ be as in Theorem \ref{1tm} and
$h_e$ be as in (\ref{Aa39}) with independent and identically
distributed $\varphi_e$ such that the bounds $c_e$, see (\ref{Aa1}),
satisfy
\begin{equation}
\label{c-e} \forall K>0 \ \forall e \qquad \mathds{E} e^{K c_e} <
\infty.
\end{equation}
Then there exists $K_*>0$ such that the set $\mathcal{G}(\gamma)$ is
almost surely a singleton whenever $K<K_*$. The unique $\mu\in
\mathcal{G}(\gamma)$ is almost surely globally Markov.
\end{theorem}
The proof of this theorem is given after the proof of Theorem
\ref{1tm} performed in the next section. Here we provide some
relative comments. Markov random fields with random interactions
naturally appear in statistical physics as states of thermodynamic
equilibrium (Gibbs states) of disordered physical systems, e.g.,
desordered magnets. A systematic presentation of the mathematical
theory of disordered systems can be found in \cite{Ant}. According
to the classification adopted herein, in Theorem \ref{2tm} we deal
with quenched Gibbs states. First works dealing with such objects
appeared in the 1980'th \cite{BD,Frol}. Then they were continued in
\cite{Berg,DKP,GM}, see also \cite{KK2} and the articles quoted in
this work. In the seminal work by Bassalygo and Dobrushin \cite{BD},
the question of uniqueness of corresponding Gibbs states on regular
lattices was studied. As a result, a uniqueness condition was
obtained by a technique based on \emph{gluing out} vertices of the
original lattice with subsequent passing to coarse-grained graphs of
a special structure, in which vertices of high degree (hubs) were
sparse. Despite the title of that paper, the technique used there is
too complicated and sometimes unclear. In \cite{KK2} we obtained a
similar result valid also for graphs with unbounded vertex degrees,
in which \emph{unboundedness} is controlled in a certain way, i.e.,
is tempered. Up to the best of our knowledge, this is the first
result of this kind obtained for Gibbs fields with binary random
interactions on unbounded degree graphs. In Theorem \ref{2tm} we
extend this result to the case of higher-order interactions.
Additionally, we obtain here the global Markov property for the
unique $\mu\in \mathcal{G}(\gamma)$. A further extension can be
dropping the condition of the identical distribution of $c_e$, and
then imposing `averaged weakness' conditions like the one in
(\ref{Gb2b}). We plan to turn to this issue in a separate work.

\section{The Proofs}

\label{SS}

 In view of Remark \ref{G1rk}, it is enough to prove
Theorem \ref{1tm} for $g_*$. Hence, below we set $g(n)=g_*(n) = \log
n$.
\subsection{Preparatory part}
We begin by making more precise our notations. By $\sf C,\sf
D\subset {\sf E}$, etc, we denote nonempty finite sets of vertices
of the line-graph ${\sf L}({\sf H})$, and also its subgraphs
generated by these sets of vertices. By ${\sf A}\subset {\sf E}$ we
always mean such a subset, for which the generated graph is
connected. Define
\begin{equation*}
{\sf B}_r(e)=\{e'\in {\sf E}: d_{\sf L}(e,e') \leq r\}, \qquad {\sf
S}_r(e)= \{e'\in {\sf E}: d_{\sf L}(e,e') = r\}.
\end{equation*}
Next, for a subset, $\sf C\subset {\sf E}$, we set
\begin{equation*}
  \langle {\sf C}\rangle = \{ x\in {\sf V}: \exists e \in {\sf C} \
  x\in e\},
\end{equation*}
that is, $ \langle {\sf C}\rangle$ is the collection of all $x\in
{\sf V}$ contained in all $e\in {\sf C}$.

For $r>0$ and a  given $x\in {\sf V}$, we fix some $e_x \in {\sf
E}_x$, and then set
\begin{gather}
  \label{Gc1}
\Lambda_{x,r} =  \langle {\sf B}_{r}(e_x) \rangle.
\end{gather}
For each $e\in {\sf S}_{r+1} (e_x)$, it follows that $e\cap
\Lambda_{x,r} \neq \varnothing$ as this $e$ has neighbors in ${\sf
B}_r(e_x)$. At the same time, if $e\in {\sf S}_q(e_x)$ with $q>
r+1$, then $e\cap \Delta_{x,r} =\varnothing$, which means that, see
(\ref{Ga8a}),
\begin{equation}
  \label{Gc3}
\partial \Lambda_{x,r} =\langle {\sf S}_{r+1}(e_x)\rangle\setminus  \Lambda_{x,r} = \bigcup_{e\in {\sf S}_{r+1}(e_x)} \left( e \cap
 \Lambda^c_{x,r} \right)  .
\end{equation}
Following \cite{F} we prove Theorem \ref{1tm} by establishing so
called \emph{strong uniqueness}, which implies both the uniqueness
in question and the global Markov property. To this end, we fix some
nonempty $\Delta \subset {\sf V}$, and also $\omega\in \Sigma$. Then
set
\begin{equation}
  \label{A}
  \mathcal{V}_\Delta = \{ \Lambda \in \mathcal{V}: \Lambda \subset
  \Delta\}.
\end{equation}
For $\sigma \in \Sigma$, by $\sigma_\Delta \times \omega_{\Delta^c}$
we mean the element of $\Sigma$ such that $(\sigma_\Delta \times
\omega_{\Delta^c})_x =\sigma_x$ for $x\in \Delta$, and
$(\sigma_\Delta \times \omega_{\Delta^c})_x =\omega_x$ for $x\in
\Delta^c$. For such fixed $\Delta$ and $\omega$, we then define
\begin{equation}
  \label{A1}
  \gamma^{\Delta,\omega}_\Lambda (A|\sigma_\Delta) = \gamma_\Lambda
  (A|\sigma_\Delta \times \omega_{\Delta^c}), \qquad A\in
  \mathcal{F}_\Delta , \quad \Lambda \in \mathcal{V}_\Delta,
\end{equation}
which is a probability kernel from $ \mathcal{S}^{\Delta \setminus
\Lambda}$ to $\mathcal{F}_\Delta$. Then
$\gamma^{\Delta,\omega}=(\gamma^{\Delta,\omega}_\Lambda)_{\Lambda
\in \mathcal{V}_\Delta}$ is a specification, which coincides with
$\gamma$ for $\Delta ={\sf V}$ and any $\omega$. By
$\mathcal{G}(\gamma^{\Delta,\omega})$ we denote the set of all
probability measures on $\mathcal{F}_\Delta$ which satisfy, cf.
(\ref{G1a}),
\begin{equation}
  \label{A2}
  \mu \gamma^{\Delta,\omega}_\Lambda = \mu, \qquad \Lambda \in
  \mathcal{V}_\Delta.
\end{equation}
\begin{lemma}[strong uniqueness]
  \label{1lm}
Let $\Delta$, $\omega$ and $\gamma^{\Delta,\omega}$ be as just
described. Then $\mathcal{G}(\gamma^{\Delta,\omega})$ is a singleton
whenever (\ref{Gb2b}) is satisfied.
\end{lemma}
The proof of this lemma is based on a certain property of the
functions $h_e$,  which  we formulate now. Let $\mathds{1}_A$ stand
for the indicator of $A\in \mathcal{S}$. For $x\in {\sf V}$ and
$A\in \mathcal{S}$, we then set $F^A_x(\sigma) =
\mathds{1}_A(\sigma_s)$. Since the functions $h_e$ are separated
away from zero, see (\ref{G7}), the following statement is a direct
consequence of \cite[Theorem 1.33, page 23]{Ge}.
\begin{proposition}
  \label{1pn}
Let $\Delta$, $\omega$ and $\gamma^{\Delta,\omega}$ be as in Lemma
\ref{1lm}. Then for each $\mu_1, \mu_2 \in
\mathcal{G}(\gamma^{\Delta,\omega})$, the equality $\mu_1(F^A_x) =
\mu_2(F^A_x)$, holding for all $x\in \Delta $ and $A\in
\mathcal{S}$, implies $\mu_1 = \mu_2$.
\end{proposition}

\subsection{The proof of the lemma}

Clearly, the lemma ought to be proved only for infinite
$\Delta\subset {\sf V}$. Let $\{\Lambda_k\}_{k\in \mathds{N}}\subset
\mathcal{V}_\Delta$, see (\ref{A}), be an ascending sequence that
exhausts $\Delta$. By (\ref{A2}), for each $x\in \Delta$ and $k$
such that $x\in \Lambda_k$, we have
\begin{gather*}
  \mu_1 (F^A_x) - \mu_2(F^A_x) = \int_{\Sigma} \int_{\Sigma} \left[ \gamma_{\Lambda_k}^{\Delta,\omega}(F^A_x|\sigma)
  - \gamma_{\Lambda_k}^{\Delta,\omega}(F^A_x|\sigma') \right] \mu_1 (d \sigma) \mu_2 ( d
  \sigma').
\end{gather*}
Then the proof will be done if, for each $x\in \Delta$, we construct
a sequence $\{\Lambda_k\}_{k\in \mathds{N}}\subset
\mathcal{V}_\Delta$ with the properties just mentioned, such that
\begin{equation}
  \label{Gc5}
 \sup_{\sigma, \sigma' \in \Sigma} \left|  \gamma^{\Delta,\omega}_{\Lambda_k} (F^A_x|\sigma) -
  \gamma^{\Delta,\omega}_{\Lambda_k }(F^A_x|\sigma')\right| \to 0, \qquad {\rm
 as} \
\  k\to +\infty.
\end{equation}
To proceed further, we recall that the family
$(h_\Lambda)_{\Lambda\in \mathcal{V}}$ is defined in (\ref{G6}) with
$(h_e)_{e\in {\sf E}}$ satisfying Assumption \ref{1ass}. In view of
this and the lower bound in (\ref{G7}), we introduce
\begin{equation}
  \label{Gc6}
  \bar{h}_e (\sigma_e) = \frac{1}{m_e} h_e (\sigma).
\end{equation}
For $r>0$, $x\in \Delta$ and $e_x\in {\sf E}_x$ as in (\ref{Gc1}),
(\ref{Gc3}), let $\mathcal{A}_{r}(e_x)$ be the corresponding
collection of animals in ${\sf L}({\sf H})$, see (\ref{Ga3}), and
$\{N_{k}\}_{k\in \mathds{N}}$ a sequence as in Definition
\ref{Ga1df}. Set
\begin{equation*}
  \Lambda_k = \Lambda_{x,N_k} \cap \Delta, \qquad \partial \Lambda_k
  = \{ z\in \Lambda_k^c : \exists e \ e\cap \Lambda_k \neq
  \varnothing \  {\rm and} \ z \in e\},
\end{equation*}
and also
\begin{equation*}
  {\sf D}^o_{k} = \{ e: e\subset \Lambda_k\}, \qquad {\sf D}_{k} = \{ e
  : e \cap \Lambda_k \neq \varnothing \ {\rm and} \ e \cap \partial \Lambda_k \neq
  \varnothing\}.
\end{equation*}
For $e\in {\sf D}_{k}$, we have the following possibilities: (a)
$d(e, e_x)\leq N_k$; (b) $d(e, e_x)= N_k+1$. Note that $d(e, e_x)>
N_k+1$ is impossible as $e\cap\Lambda_{x,N_k} \neq \varnothing$. In
case (a), each $z\in e\setminus \Lambda_k$ lies in $\Delta^c$, as it
lies in $\Lambda_{x,N_k}$ and $z\in \Delta$ would imply $z\in
\Lambda_k$. In case (b), $z\in e$ can be in $\Lambda_k$,
$\Delta\setminus \Lambda_k$ and $\Delta^c$. Thus, for $e\in {\sf
D}_{k}$, we write $e = e_1 \cup e_2 \cup e_3$: $e_1=e\cap
\Lambda_k$, $e_2 = e\cap (\Delta \setminus \Lambda_k)$ and
$e_3=e\cap \Delta^c$. In accord with this, we also write
\begin{gather}
  \label{A11}
{\sf D}_{k} = {\sf D}_{2,k} \cup {\sf D}_{3,k}, \\[.2cm] \nonumber
{\sf D}_{2,k} =\{e\in {\sf D}_k:  e_2= e\cap (\Delta \setminus
\Lambda_k)\neq \varnothing \}, \quad {\sf D}_{3,k} = \{e\in {\sf
D}_k:  e_2= e\cap (\Delta \setminus \Lambda_k)= \varnothing \}.
\end{gather}
According to the analysis made above (case (b)), it follows that
\begin{equation}
  \label{A10}
{\sf D}_{2,k} \subset {\sf S}_{N_k+1}(e_x).
\end{equation}
Now by (\ref{A1}) we have
\begin{gather}
  \label{Gc7}
\gamma^{\Delta,\omega}_{\Lambda_k} (F^A_x|\sigma) =
\frac{1}{Z^{\Delta,\omega}_{\Lambda_{k}}(\sigma)} \int_{S^{
\Lambda_{k} }}
\mathds{1}_A(\xi_x) \prod_{e\in {\sf D}^o_{k}} \bar{h}_e (\xi_{e}) \prod_{e\in {\sf D}_{3,k}} \bar{h}_e (\xi_{e_1} \times \omega_{e_3}) \\[.2cm]
\nonumber \times \prod_{e\in{\sf D}_{2,k}} \bar{h}_e (\xi_{e_1}
\times \sigma_{e_2}\times \omega_{e_3}) \chi^{\Lambda_{k}} (d
\xi_{\Lambda_{k}}).
\end{gather}
Here $\xi_{e_1} \times \sigma_{e_2} \times \omega_{e_3}$ is the
configuration $\eta$ such that: $\eta_z = \xi_z$, $z\in e_1$;
$\eta_z = \sigma_z$, $z\in e_2$; $\eta_z = \omega_z$, $z\in e_3$.
And also
\begin{gather}
\label{Gc8} Z^{\Delta,\omega}_{\Lambda_k}(\sigma):= \int_{S^{
\Lambda_{k} }}\prod_{e\in {\sf D}^o_{k}} \bar{h}_e (\xi_{e'}) \prod_{e\in {\sf D}_{3,k}} \bar{h}_e (\xi_{e_1} \times \omega_{e_3}) \\[.2cm]
\nonumber \times \prod_{e\in{\sf D}_{2,k}} \bar{h}_e (\xi_{e_1}
\times \sigma_{e_2}\times \omega_{e_3}) \chi^{\Lambda_{k}} (d
\xi_{\Lambda_{k}}).
\end{gather}
To optimize our notations, for $e\in {\sf D}_k^o\cup {\sf D}_{3,k}$
by $\xi^\omega_e$ we denote $\xi_e$ if $e\in {\sf D}_k^o$, and
$\xi_{e_1}\times \omega_{e_3}$ if $e\in {\sf D}_{3,k}$. And also
\begin{equation*}
  \breve{\sf D}_k := ( {\sf D}_k^o \cup {\sf D}_{3,k})\setminus
  \{e_x\}.
\end{equation*}
Then by means of (\ref{Gc7}) we write
\begin{eqnarray}
  \label{Gc9}
M_{x,N_k} (A) & := & \gamma^{\Delta,\omega}_{\Lambda_{k}}
(F^A_x|\sigma) - \gamma^{\Delta,\omega}_{\Lambda_{k}}(F^A_x|\sigma')
= \frac{1}{
Z^{\Delta,\omega}_{\Lambda_{k}}(\sigma) Z^{\Delta,\omega}_{\Lambda_{k}}(\sigma')}\\[.2cm]
\nonumber & \times & \int_{S^{ \Lambda_{k} }}\int_{S^{ \Lambda_{k}
}} \left[\mathds{1}_A(\xi_x) - \mathds{1}_A(\eta_x) \right]
\bar{h}_{e_x} (\xi^\omega_{e_x})\bar{h}_{e_x} (\eta^{\omega}_{e_x})
\prod_{e\in
\breve{\sf D}_{k} } \left(1 + \Gamma_e (\xi^\omega,\eta^\omega) \right) \\[.2cm]
\nonumber & \times & \Psi^\omega(\xi, \sigma; \eta, \sigma')
\chi^{\Lambda_{k}} (d \xi_{\Lambda_{k}}) \chi^{\Lambda_{k}} (d
\eta_{\Lambda_{k}}),
\end{eqnarray}
where
\begin{equation}
  \label{Gc10}
\Gamma_e (\xi^\omega,\eta^\omega) := \bar{h}_e
(\xi^\omega_e)\bar{h}_e (\eta^\omega_e) -1 \geq 0,
\end{equation}
see (\ref{Gc6}), and
\begin{equation}
  \label{Gc11}
\Psi^\omega(\xi, \sigma; \eta, \sigma') := \prod_{e\in {\sf
D}_{2,k}} \bar{h}_e (\xi_{e_1} \times \sigma_{e_2} \times
\omega_{e_3}) \bar{h}_e (\eta_{e_1} \times \sigma'_{e_2} \times
\omega_{e_3}).
\end{equation}
Now we observe that the term in the second line of (\ref{Gc9}) is
anti-symmetric with respect to the interchange $\xi \leftrightarrow
\eta$. Keeping this in mind we rewrite the product in this line in
the following form
\begin{gather}
  \label{Gc12}
\prod_{e\in \breve{\sf D}_{k} } \left(1 + \Gamma_e
(\xi^\omega,\eta^\omega) \right) = \sum_{{\sf C}\subset \breve{\sf
D}_{k}} \Gamma_{\sf
C} (\xi^\omega,\eta^\omega), \\[.2cm]
\nonumber \Gamma_{{\sf C}} (\xi^\omega,\eta^\omega) := \prod_{e\in
{\sf C}} \Gamma_{e} (\xi^\omega,\eta^\omega).
\end{gather}
Then set, cf. (\ref{A10}),
\begin{equation}
  \label{A12}
  {\sf D}^{-}_{2,k} =\{ e\in {\sf S}_{N_k} (e_x): e\sim e' \ {\rm
  for} \ {\rm some} \ e'\in {\sf D}_{2,k}\}.
\end{equation}
Let $\mathcal{C}$ denote the collection of subsets ${\sf C}\subset
\breve{\sf D}_{k}$ which satisfy: the graph ${\sf C}\cup\{e_x\}$ has
a connected component, say $\sf A$, such that ${\sf A} \cap {\sf
D}^{-}_{2,k}\neq \varnothing$. Now we plug the first line of
(\ref{Gc12}) in (\ref{Gc9}) and observe that the integral therein is
nonzero only whenever the sum is taken over ${\sf C}\in\mathcal{C}$,
which follows by the anti-symmetricity mentioned above. Indeed, the
just mentioned condition means that there exist $z\in e\cap e'$ with
$e\in {\sf D}^{-}_{2,k}$ and $e' \in {\sf D}_{2,k}$. If this $z$
lies in $\Delta$, then the corresponding $\xi_z$ and $\eta_z$ are
present in both $\Gamma_{\sf C}(\xi^\omega, \eta^\omega)$ and
$\Psi^\omega(\xi,\sigma;\eta,\sigma')$, see (\ref{Gc11}) and
(\ref{A11}), which destroys the mentioned anti-symmetricity.
Furthermore, the graph ${\sf A}$ contains a path, $\vartheta (e_x,
e)$, connecting $e_x$ to a certain $e\in {\sf S}_{N_k}(e_x)$, see
(\ref{A12}); hence,  $|\vartheta (e_x,e)|=:N \geq N_k$, which means
that this path belongs to $\varTheta_{N_k}^N(e_x)$ for some $N\geq
N_k$, see (\ref{Ga33}). Set
\begin{equation*}
\varTheta_{k} = \{ \vartheta (e_x, e) :e \in {\sf D}^{-}_{2,k}\} .
\end{equation*}
Recall that ${\sf A}_\vartheta$ denotes the subgraph generated by
the vertices of $\vartheta$. Since some of ${\sf C}\in \mathcal{C}$
may contain several paths $\vartheta (e_x,e)\in \varTheta_k$, it
follows that
\begin{equation}
  \label{Gc15}
  \sum_{{\sf C}\in\mathcal{C}}  \Gamma_{\sf C} (\xi,\eta) \leq  \sum_{\vartheta \in \varTheta_k} \Gamma_{{\sf A}_\vartheta}
   (\xi^\omega,\eta^\omega)
  \sum_{{\sf C} \subset \breve{\sf D}_{k} \setminus {\sf A}_\vartheta} \Gamma_{{\sf C}}
  (\xi^\omega,\eta^\omega),
\end{equation}
for certain ${\sf A}_\vartheta$ may appear twice on the right-hand
side of (\ref{Gc15}): once in the first sum, and then as a subset of
${\sf C}$. We apply all these arguments in (\ref{Gc9}) and obtain
\begin{eqnarray}
  \label{Gc16}
  \left|M_{x,N_k} (A) \right| & \leq & \frac{2}{
Z^{\Delta,\omega}_{\Lambda_k}(\sigma)
Z^{\Delta,\omega}_{\Lambda_k}(\sigma')}\int_{S^{ \Lambda_{k}
}}\int_{S^{
\Lambda_{k} }} \bar{h}_{e_x}(\xi^\omega_{e_x})\bar{h}_{e_x}(\eta^\omega_{e_x})\\[.2cm] \nonumber &
\times &\left( \sum_{{\sf C} \in \mathcal{C}} \Gamma_{\sf C}
(\xi^\omega,\eta^\omega) \right) \Psi^\omega(\xi, \sigma; \eta,
\sigma') \chi^{\Lambda_{k}} (d \xi_{\Lambda_{k}}) \chi^{\Lambda_{k}}
(d \eta_{\Lambda_{k}}) \\[.2cm] \nonumber & \leq & \frac{2}{
Z^{\Delta,\omega}_{\Lambda_k}(\sigma)
Z^{\Delta,\omega}_{\Lambda_k}(\sigma')}\int_{S^{ \Lambda_{k}
}}\int_{S^{
\Lambda_{k} }} \bar{h}_{e_x}(\xi^\omega_{e_x})\bar{h}_{e_x}(\eta^\omega_{e_x})\\[.2cm] \nonumber &
\times & \left(\sum_{\vartheta\in \varTheta_k } \Gamma_{{\sf
A}_\vartheta} (\xi^\omega,\eta^\omega) \right)\left( \sum_{{\sf C}
\subset \breve{\sf D}_{k}\setminus {\sf A}_\vartheta} \Gamma_{{\sf
C}} (\xi^\omega,\eta^\omega) \right)
\Psi^\omega(\xi, \sigma; \eta, \sigma') \\[.2cm] \nonumber &
\times & \chi^{\Lambda_{k}} (d \xi_{\Lambda_{k}}) \chi^{\Lambda_{k}}
(d \eta_{\Lambda_{k}}).
\end{eqnarray}
The next step is to estimate the first two multipliers in the
penultimate line of (\ref{Gc16}). By (\ref{G8a}) and (\ref{Gc6}) it
follows that
\[
\bar{h}_e (\xi_e) \leq e^{\delta (e)}, \qquad {\rm for} \ {\rm all}
\ \xi_e \in S^e,
\]
which means that, see (\ref{Gc10}),
\begin{eqnarray}
  \label{Gc18}
  \Gamma_e (\xi^\omega, \eta^\omega)& \leq & e^{2 \delta (e)} -1 .
\end{eqnarray}
Since each ${\sf A}$ is in $\mathcal{A}_{N_k}(e_x)$, by
(\ref{Gc18}), (\ref{Gb2a}) and (\ref{Gb2b}) we get, see Definition
\ref{BBdf},
\begin{gather}
  \label{Gc19}
\sum_{\vartheta\in \varTheta_k } \Gamma_{{\sf A}_\vartheta}
(\xi^\omega,\eta^\omega) = \sum_{N=N_k}^\infty \sum_{\vartheta \in
\varTheta_{N_k}^N(e_x)}\prod_{e\in {\sf A}_\vartheta} \Gamma_e
(\xi^\omega,\eta^\omega)\\[.2cm] \nonumber \leq \sum_{N=N_k}^\infty  \sum_{\theta \in \varTheta_{N_k}^N(e_x)}
 \exp \bigg{(}N \mathfrak{D}(\vartheta)  \bigg{)} \leq \sum_{N=N_k}^\infty e^{-\epsilon} =
\frac{e^{-\epsilon N_k}}{e^\epsilon -1},
\end{gather}
where we have used also (\ref{Ga5}). At the same time
\begin{gather}
  \label{Gc20}
  \sum_{{\sf C} \subset
\breve{\sf D}_{k}\setminus {\sf A}_\vartheta} \Gamma_{{\sf C}}
(\xi^\omega,\eta^\omega) \leq \sum_{{\sf C} \subset \breve{\sf
D}_{k}} \Gamma_{{\sf C}} (\xi^\omega,\eta^\omega) = \prod_{e\in
\breve{\sf D}_{k}} \bar{h}_e (\xi^\omega_e)
\bar{h}_{e}(\eta^\omega_e).
\end{gather}
Now we use (\ref{Gc20}), (\ref{Gc19}) in (\ref{Gc16}), take into
account (\ref{Gc8}) and arrive at
\begin{equation*}
  \left|M_{x,N_k} (A) \right|  \leq  2(e^\epsilon -1)^{-1} e^{- \epsilon N_k},
\end{equation*}
which yields (\ref{Gc5}) and thus completes the proof.

\subsection{The proof  of the theorems}

We begin by proving our main statement.
\newline \noindent {\it Proof of Theorem \ref{1tm}.}
The uniqueness in question follows as a particular case of Lemma
\ref{1lm} corresponding to $\Delta ={\sf V}$. The proof of the
global Markov property of the unique $\mu \in \mathcal{G}(\gamma)$
follows by F\"ollmer's arguments \cite[pages 266, 267]{F}, which we
repeat here for the reader's convenience.

Since $(\Sigma, \mathcal{F})$ is a standard Borel space, for $\mu
\in \mathcal{G}(\gamma)$ one may get
\begin{equation}
  \label{A4}
\mu^{\Delta^c,\omega} (A) = \mu(A|\mathcal{F}_{\Delta^c})(\omega) ,
\quad \mu^{\partial \Delta,\omega} (A) = \mu(A|\mathcal{F}_{\partial
\Delta})(\omega) , \qquad A \in \mathcal{F}_\Delta.
\end{equation}
Let $F:\Sigma \to \mathds{R}$ and $G_\Lambda:\Sigma \to \mathds{R}$,
$\Lambda \in \mathcal{V}_\Delta$, be bounded and: $F$ is
$\mathcal{F}_{\Delta^c}$-measurable; $G_\Lambda$ is
$\mathcal{F}_{\Lambda}$-measurable. Then
\begin{gather}
  \label{A5}
  \mu(F G_\Lambda) = \int_{\Sigma} F(\omega)
  \mu^{\Delta^c,\omega}(G_\Lambda) \mu ( d \omega).
\end{gather}
At  the same time, by (\ref{A1}) we have
\begin{gather}
  \label{A6}
  \int_{\Sigma}F(\omega)\left(( \mu^{\Delta^c,\omega} \gamma^{\Delta,\omega}_\Lambda )(G_\Lambda)
  \right) \mu(d\omega) =   \int_{\Sigma}F(\omega)\left((\mu^{\Delta^c,\omega} \gamma_\Lambda)
  (G_\Lambda)\right)(\omega) \mu(d \omega) \\[.2cm] \nonumber =  \int_{\Sigma}F(\omega)\left((\mu^{\Delta^c,\omega}
  (G_\Lambda)\right)(\omega) \mu(d \omega) = \mu(F G_\Lambda).
\end{gather}
As $F$ one may take any $\mathcal{F}_{\Delta^c}$- measurable
functions; thus, (\ref{A5}) and (\ref{A6}) imply that
$\mu^{\Delta^c,\omega}$ satisfies (\ref{A2}), which yields
$\mu^{\Delta^c,\omega}\in \mathcal{G}(\gamma^{\Delta^c,\omega})$.
However, by repeating the same steps with $\mathcal{F}_{\partial
\Delta}$-measurable functions $F$ one gets that $\mu^{\partial
\Delta,\omega}\in \mathcal{G}(\gamma^{\Delta^c,\omega})$, which by
(\ref{A4}) gives the global Markov property in question. \hfill
$\square$
\newline \noindent {\it Proof of Theorem \ref{2tm}.}
For random $h_e$, the quantity introduced in (\ref{Gc9}) is a random
variable. Thus, for  fixed $x$ and $A$, we have a sequence of random
variables
\begin{equation}
  \label{Ab1} X^K_k = |M_{x,N_k}(A)|,
\end{equation}
where we indicate the dependence on $K$. Recall that $\delta(e) = K
c_e$ in this case. Our aim is to show that this sequence is almost
surely asymptotically degenerate at zero for each $x$ and $A$. By
(\ref{Gc16}) we get
\begin{gather}
  \label{Ab2} X_k^K \leq \sum_{\vartheta \in \varTheta_k}
  \prod_{e\in {\sf A}_\vartheta} \left(e^{2K c_e}-1 \right).
\end{gather}
According to the assumed i.i.d. of $c_e$, and also to (\ref{c-e}),
the map
\[
K \mapsto \tau(K):= \mathds{E}\left( e^{2Kc_e }-1\right)
\]
is increasing and continuous (by the dominated convergence theorem),
and satisfies $\tau(0) =0$. Then either $\tau(K) < e^{-\bar{a}}$ for
all $K>0$, or there exists a unique $K_*>0$ such that
$\tau(K_*)=e^{-\bar{a}}$ and $\tau(K)<e^{-\bar{a}}$ for $K< K_*$.
Assuming $K_* =+\infty$ in the former case, we have that the
following holds, see (\ref{Gc19}), (\ref{Ab2}) and Remark
\ref{G1rk},
\begin{gather*}
 \mathds{E}X^K_k \leq \sum_{N=N_k}^\infty \sum_{\vartheta \in
 \varTheta^N_{N_k}(e_x) } \left[ \tau(K)\right]^N \leq
 \sum_{N=N_k}^\infty  \left[e^{\bar{a}} \tau(K)\right]^N = \frac{[e^{\bar{a}} \tau(K)]^{N_k}}{1-e^{\bar{a}}
 \tau(K)}, \qquad K< K_*.
\end{gather*}
Hence, the sequence $\{X_k^K\}$  is asymptotically degenerate at
zero in mean. Then it contains a subsequence, which converges to
zero almost surely, see, e.g., \cite[Theorem 3.4]{Gut}. In view of
Proposition \ref{1pn}, this yields the uniqueness in question. The
almost sure global Makovianess of the unique $\mu$ is then obtained
as in the prof of Theorem \ref{1tm}. \hfill $\square$.



\section*{Acknowledgment}
The second author was supported in part by the Deutsche
Forschungsgemeinschaft (DFG) through the SFB 1238 "Taming
uncertainty and profiting from randomness and low regularity in
analysis, statistics and their applications".

v

\end{document}